\documentclass{article}
\usepackage[]{amsmath,amssymb}
\usepackage[]{amsthm,enumerate}
\usepackage{verbatim,bm}
\usepackage{color}

\newtheorem{theorem}{Theorem}[section]
\newtheorem{proposition}[theorem]{Proposition}
\newtheorem{lemma}[theorem]{Lemma}
\newtheorem{corollary}[theorem]{Corollary}
\theoremstyle{definition}

\newtheorem{example}[theorem]{Example}

\theoremstyle{remark}
\newtheorem{remark}[theorem]{Remark}

\makeatletter
\@addtoreset{equation}{section}

\makeatother

\begin{document}

\title{Weighing matrices and spherical codes}
\date{\today}

\author{
Hiroshi Nozaki, Sho Suda \\ {\small Department of Mathematics Education,  Aichi University of Education}\\ {\small 1 Hirosawa, Igaya-cho, Kariya, Aichi 448-8542, Japan} \\ {\small \texttt{hnozaki@auecc.aichi-edu.ac.jp}, \texttt{suda@auecc.aichi-edu.ac.jp}}}

\maketitle
\begin{abstract}
Mutually unbiased weighing matrices (MUWM) are closely related to an antipodal spherical code with 4 angles.  
In the present paper, we clarify the relationship between MUWM and the spherical sets, and give the complete solution about the maximum size of a set of MUWM of weight 4 for any order.  Moreover we describe some natural generalization of a set of MUWM from the viewpoint of spherical codes, and determine several maximum sizes of the generalized sets. They include an affirmative answer 
of the problem  of  Best, Kharaghani, and  Ramp. 
\end{abstract}
\section{Introduction}
A weighing matrix $W$ of weight $k$ is a square $(\pm1,0)-$matrix $W$ of order $d$ such that $W W^T=k I$, where $I$ is the identity matrix of order $d$ and $W^T$ denotes the transpose of $W$. 
If $k=d$ holds, a weighing matrix is a Hadamard matrix.  The set of row vectors in $W$ is identified with a finite set on a sphere of dimension $d$ where
distinct two vectors are orthogonal. The spherical code equivalently gives the vertices of a cross-polytope.     
From the viewpoint of this relationship,  a set of mutually unbiased weighing matrices (MUWM), which is introduced in \cite{HKO}, is identified with an antipodal spherical code (or equivalently lines in $\mathbb{R}^d$) with only 4 angles which is decomposed into cross-polytopes.  
In the present paper, we determine the maximal sizes of sets of some MUWM or their generalized objects by the construction of spherical codes and related linear $\mathbb{Z}_2$, $\mathbb{Z}_4$-codes. The results include an affirmative answer to Conjecture 23 in \cite{BKR}. 

A weighing matrix is extensively studied as a generalization of a Hadamard matrix in combinatorics. 
The weighing matrices were classified for small orders and weights \cite{CRS86,O89,O92,O93,OM,HM}.  
Note that the classification of self-orthogonal codes is used to classify weighing matrices in \cite{HM}. 
 As a generalization of mutually unbiased bases, Holzmann, Kharaghani and Orrick defined mutually unbiased weighing matrices (MUWM) in \cite{HKO}.
Recently Best, Kharaghani, and Ramp studied MUWM further in \cite{BKR}. In particular they obtained the maximum numbers of MUWM of order $n$ and weight $w$ for $(n,w)=(7,4),(8,4)$. 
Actually the two maximum examples correspond to the roots in 
the $E_7$, $E_8$ lattices, but the authors did not mention it.  
This is one of motivations to study MUWM from the viewpoint of spherical codes. 

In this paper, we put forward a new concept, that is a set of mutually quasi-unbiased weighing matrices (MQUWM), 
it is a natural generalization of MUWM from the perspective of 
 antipodal spherical codes with 4 angles. 
As well as being a natural generalization,  MQUWM make MUWM (see Theorem~\ref{thm:mquwm}), and they are important 
objects relating to the existence of original MUWM. 
We prove that the existence of some spherical codes is equivalent to that of MQUWM, and we show the construction of the spherical codes from linear $\mathbb{Z}_2,\mathbb{Z}_4$-codes or root lattices. 
  As a consequence we obtain several maximum numbers of 
MQUWM or MUWM.  
  
The following is a summary of the paper. 
In Section~\ref{sec:2}, we prepare some basic definitions and results of MUWM, linear codes, and root lattices. 
In Section~\ref{sec:3}, we define MQUWM, and  give the relationship between MQUWM and spherical codes.
In Section~\ref{sec:code}, we determine the maximum number of MQUWM for some parameters by linear codes. 
They imply the answer of Conjecture 23 in \cite{BKR}. 
In Section~\ref{sec:4},  we give the maximum number of MQUWM for some basic parameters. 
In Section~\ref{subsec:frame}, we give the complete solution about the maximum size of a set of MUWM of weight 4 by the cross-polytope decomposition of the roots in root lattices.

%%%%%%%%%%%%%%%%%%%%%%%%%%%%%%%%%%%%%%%%%%%%%
\section{Preliminaries} \label{sec:2}
In this section, we introduce some basic definitions used in this paper, and results about disjoint $2$-frames in a root lattice. 

A square $(\pm1,0)-$matrix $W$ of order $d$ is a {\it weighing matrix of weight $k$} if $W W^{T}=k I$ holds.
A weighing matrix of order $d$ and weight $d$ is known as a Hadamard matrix and that of order $d$ and  weight $d-1$ is known as a conference matrix.

Two weighing matrices $W_1,W_2$ of order $d$ and weight $k$ are said to be {\it unbiased} if $(1/\sqrt{k})W_1 W_2^{T}$ is also 
weighing matrix of weight $k$ \cite{HKO}.
A set of weighing matrices $\{W_1,\ldots,W_f\}$ is said to be {\it mutually unbiased weighing matrices} (MUWM)  if any distinct two of them are unbiased.
% Holzmann, Kharaghani and Orrick \cite{HKO} provided a construction of MUWM using mutually suitable Latin squares, which are equivalent to mutually orthogonal Latin squares, and  
% Best, Kharaghani and Ramp \cite{BKR} studied MUWM and provided the bounds on the number of MUWM.

To construct mutually quasi-unbiased weighing matrices, which are defined later, we prepare two codes over $\mathbb{Z}_2$ or $\mathbb{Z}_4$. 

Let $\mathcal{B}(2,m)$ be a cyclic code of length $2^m-1$ with defining set $C_1\cup C_2\cup C_3\cup C_4$, where $C_i$ is the $2$-cyclotomic coset of $i$ modulo $2^m-1$ for $i=1,2,3,4$.
The code $\mathcal{B}(2,m)$ is called the {\it narrow-sense BCH code} with designed distance $5$.
Then the dual code $\mathcal{B}(2,m)^\perp$ has the weights $\{0,2^{m-1}-2^{(m-1)/2},2^{m-1},2^{m-1}+2^{(m-1)/2}\}$ \cite[Table 11.2]{HV}.
Let $C$ be a code generated by the extended code of $\mathcal{B}(2,m)^\perp$ and the all-ones vector.
Then $C$ contains the first order Reed-Muller code of length $2^m$ as a subcode and the set of its weights is $\{0,2^{m-1}-2^{(m-1)/2},2^{m-1},2^{m-1}+2^{(m-1)/2},2^m\}$.

Let $h(x)$ be a primitive basic irreducible polynomial of degree $m$,  and $g(x)$ be the reciprocal polynomial to $(x^{2^m-1}-1)/((x-1)(h(x)))$.  
Let $\mathcal{K}'(m)$ be a code of length $2^m-1$ over $\mathbb{Z}_4$ with generator polynomial $g(x)$. 
A $\mathbb{Z}_4$-Kerdock code $\mathcal{K}(m)$ is the extended code of length $d=2^m$ obtained by adding an overall parity check to $\mathcal{K}'(m)$.
For $i\in\{0,1,2,3\}$ and $x\in \mathbb{Z}_4^d$, let $n_i(x)$ denote the number of coordinates equal to $i$. 
Then the set of  $(n_0(x)-n_2(x))+\sqrt{-1}(n_2(x)-n_4(x))$  for $x\in \mathcal{K}(m)$ is 
$\{\pm d,\pm\sqrt{-1}d,0,(\pm\sqrt{d}\pm\sqrt{-d})/\sqrt{2}\}$ if $m$ is odd and $\{\pm d,\pm\sqrt{-1}d,0,\pm\sqrt{d},\pm\sqrt{-d}\}$ if $m$ is even, see \cite{HKCSS}, \cite[Section 12]{HV}, \cite{BD} for more details .

An integral lattice is called a {\it root lattice} if it is generated by roots, which are vectors whose norm $\sqrt{2}$. It is well known that the irreducible root lattices are $A_d (d\geq 2)$, $D_d (d\geq 4)$, $E_6$,
$E_7$, and $E_8$ see for example \cite{E}. 
The subset $F=\{\pm v_1,\ldots, \pm v_d\}$ in a lattice of rank $d$ is called a {\it $k$-frame} if 
the usual inner products satisfy $(v_i,v_j)=k\delta_{ij}$, 
where $\delta_{ij}$ is the Kronecker delta. 

In Section~\ref{subsec:frame}, we describe the relationship between MUWM of weight 4 and  disjoint $2$-frames in roots. We prepare several results about disjoint $2$-frames in a root lattice. Let $\perp_i L_i$ denote the orthogonal direct sum of lattices $L_i$.  Let $m(\Lambda)$ be the maximum number of disjoint $2$-frames in  a lattice $\Lambda$.

\begin{lemma} \label{lem:red}
Let $L=\perp_i L_i$, where $L_i$ is an irreducible root lattice.  Then we have
\[
m(L) = {\rm min}_i \{ m(L_i)\}.
\]
\end{lemma}
\begin{proof}
For each $v \in L$, we have 
\begin{equation}
(v,v)= (\sum_i  v_i, \sum_i v_i) = \sum_i(v_i,v_i ),  \label{eq:(v,v)}
\end{equation}
where $v_i \in L_i$. When $v$ is a root in $L$, 
it clearly holds that there exists a root $v_j \in L_j$ such that 
$v=v_j$ by \eqref{eq:(v,v)} and $(v_i,v_i)\geq 2$ for each $i$.   
Therefore the number of disjoint $2$-frames
of a reducible lattice is bounded above by that of each irreducible components. Moreover each $L_j$ has the number ${\rm min}_i \{ m(L_i)\}$ of disjoint $2$-frames, and hence we can construct the number ${\rm min}_i \{ m(L_i)\}$ of disjoint $2$-frames in $L$. 
\end{proof}
We use the notation 
\[
(a_1, \ldots, a_d)^P=\{(a_{\sigma(1)}, \ldots, a_{\sigma(n)} ) \mid 
\sigma \in S_d \},  
\]
for $(a_1, \ldots, a_d) \in \mathbb{R}^d$, where $S_d$ is the symmetric group of degree $d$.  
By Lemma~\ref{lem:red}, the maximum number of disjoint 2-frames in an irreducible root lattice is essential. The following lemma shows the number.  
\begin{lemma} \label{lem:maxframes}
We have the following. 
\begin{enumerate}
\item For any $d\geq 2$, $m(A_d)=0$.
\item For even $d\geq 4$,  $m(D_d)= d-1$. 
\item For odd $d\geq 5$,  $m(D_d)=0$.
\item $m(E_6)=0$.
\item $m(E_7)=9$.
\item $m(E_8)=15$.
\end{enumerate}
\end{lemma}
\begin{proof} 
(1)  The set of roots in the $A_d$ lattice can be expressed by  $(1,-1,0,\ldots,0)^P$ which has size $d(d+1)$.
Clearly the largest number of mutually orthogonal vectors is 
$\lfloor (d+1)/2 \rfloor$.   Therefore $(1)$ follows.

(2)  The set of roots in the $D_d$ lattice can be expressed by $(\pm 1, \pm 1, 0, \ldots, 0)^P$.  
%a perfect maching $M$ of the complete graph $K_{d}$.  
Disjoint 2-frames in the $D_d$ lattice are related to disjoint perfect matchings of the complete graph $K_d$. 
A matching is a set of pairwise non-adjacent edges. 
A matching is said to be perfect if every vertex is an endpoint of some edge in the matching.  
For a perfect matching $M$ of $K_d$, we can obtain a 2-frame   
\[
F_M=\{ (v_1,\ldots,v_d) \in D_d \mid \sum_k v_k^2= 2,  v_i= \pm 1, v_j=\pm 1, \{i,j\} \in M  \}. 
\] 
Since $d$ is even, the complete graph $K_{d}$ is $1$-factorable, that is, 
the edge set $E(K_d)$ can be decomposed into $d-1$ disjoint perfect matchings \cite[Theorem 9.1]{H}. This implies that 
a set of roots in the $D_d$ lattice is decomposed into $d-1$ disjoint $2$-frames. Therefore (2) follows.

(3) For odd $d$,  the largest number of mutually orthogonal roots in the $D_d$ lattice is clearly $d-1$. Therefore (3) follows.

(4) We can show (4) by a exhaustive computer search.

(5) There exists  a set of $8$ MUWM of weight $4$ and order $7$ \cite{BKR}. Adding $(\pm 2,0,\ldots,0)^P$ to a set of row vectors in them, we can construct $126$ roots that must come from the roots of $E_7$. 
This implies (5).

(6) There exists a set of $14$ MUWM of weight $4$ and order $8$ \cite{BKR}. Adding $(\pm 2,0,\ldots,0)^P$ to a set of row vectors in them, we can construct $240$ roots that must come from the roots of $E_8$. 
This implies (6). 
\end{proof}

%%%%%%%%%%%%%%%%%%%%%%%%%%%%%%%%%%%%%%%%%%%%%%%%%%%%%%%%%%%%%%%%%%%%%%%%%%%%%%%%%%%%%%%%
\section{Mutually quasi-unbiased weighing matrices} \label{sec:3}
In this section we introduce the concept of mutually quasi-unbiased weighing matrices (MQUWM) with connection to MUWM. 
We also show that the existence of MQUWM is equivalent to 
that of some spherical finite set. 

Two weighing matrices $W_1,W_2$ of order $d$ and weight $k$ are said to be  
{\it quasi-unbiased for parameters $(d,k,l,a)$} if there exist positive integers $a,l$ such that $(1/\sqrt{a}) W_1 W_2^T$ is a weighing matrix of weight $l$. 
 Since $(1/\sqrt{a}) W_1 W_2^T$ is a weighing matrix of weight $l$ for quasi-unbiased weighing matrices $W_1,W_2$, it holds that $l=k^2/a$. 
Weighing matrices $W_1,\ldots,W_f$ are said to be {\it mutually quasi-unbiased weighing matrices (MQUWM) for parameters $(d,k,l,a)$}  if any distinct two of them are quasi-unbiased  for parameters $(d,k,l,a)$.

A set of MQUWM for parameters $(d,d,d,d)$ coincides with a set of  mutually unbiased bases (MUB) in $\mathbb{R}^d$, and that for parameters $(d,k,k,k)$ coincides with a set of MUWM of order $d$ and weight $k$. 
Thus the concept of MQUWM is a simultaneous generalization of both MUB and MUWM. 

The following theorem shows that a set of MQUWM is an important concept related to the existence of MUWM. 

\begin{theorem}\label{thm:mquwm}
Suppose $\{W_1,\ldots, W_f\}$ is a set of mutually quasi-unbiased weighing matrices for parameters $(d,k,l,a)$.  
Then $\{(1/\sqrt{a})W_2 W_1^T,\ldots, (1/\sqrt{a})W_f W_1^T \}$ is a set of mutually unbiased weighing matrices of weight $l$.
\end{theorem}
\begin{proof}
We can prove the theorem by a direct calculation.
\end{proof}
%%As a consequence of Theorem~\ref{thm:mquwm} we have the following corollary. 
\begin{corollary}\label{cor:muwm}
If $\{W_1,\ldots, W_f\}$ is a set of mutually unbiased weighing matrices of weight $k$, 
then  $\{W_1^T, (1/\sqrt{k})W_2 W_1^T,\ldots, (1/\sqrt{k})W_f W_1^T \}$ is a set of mutually unbiased weighing matrices of weight $k$.
\end{corollary}

If there exist MQUWM, then we obtain MUWM. Note that we have 
MUWM which cannot be obtained from MQUWM which is not MUWM. The following is  a such example. 
\begin{example}\label{ex:1}
There exists a set of $8$ MUWM of weight $4$ and order $7$ \cite{BKR}. They do not come from MQUWM in the sense of Theorem~\ref{thm:mquwm}.  
Since $k \leq 7$ holds, we have 
$4a=k^2\leq 49$, and $a=1,4,9$. For $(k,a)=(2,1)$, there
does not exist a set of $9$ MQUWM of weight $2$ by Theorem~\ref{thm:bound2}.
For $(k,a)=(4,4)$, this case corresponds to Corollary~\ref{cor:muwm}.
For $(k,a)=(6,9)$, there does not exit a weighing matrix of 
weight $6$ and order $7$, because if the order is odd, then the weight must be square \cite{CRS86}. 
Therefore we do not have corresponding MQUWM. 
\end{example}

Let $r S^{d-1}$ denote the unit sphere in $\mathbb{R}^d$ whose radius is $r$.
For a finite set $X$ of $rS^{d-1}$, let $A(X)$ be the set of 
usual inner products of two distinct vectors in $X$.
We say $\{X_0,X_1,\ldots,X_f\}$ is  a {\it cross polytope decomposition} of $X$ if elements of $X_i$ consist of vectors of  a cross polytope for each $i\in\{0,1,\ldots,f\}$ and $\{X_0,X_1,\ldots,X_f\}$ is a partition of $X$.
Let $\Omega_{d}=\{0,\pm1\}^d$ and $\Omega_{d,k}=\{x\in \Omega_d\mid \sum_i x_i^2=k\}$.
For a matrix $A$, denote by $S(A)$ the set of row vectors of $A$.
 
The following proposition characterizes the existence of MQUWM in terms of that of  certain spherical codes. 
\begin{proposition}\label{prop:MQUWM}
Let $f$, $d$, $k$, $a$ be positive integers such that $f\geq2$. 
The existence of the following are equivalent.
\begin{enumerate}
\item a set $\{W_1,\ldots,W_f\}$ of mutually quasi-unbiased weighing matrices for parameters $(d,k,k^2/a,a)$,  
\item a nonempty subset $X\subset  \Omega_{d,k}$ with the property that $A(X)=\{\pm \sqrt{a} ,0,-k\}$ and there exists a cross polytope decomposition $\{X_1,\ldots,X_f\}$ of $X$.
\end{enumerate}
\end{proposition}
\begin{proof}
(1)$\Rightarrow$(2): 
Let  $X_i=S(W_i)\cup S(-W_i)$ for $i=1,\ldots,f$. 
Then $X=\cup_{i=1}^f X_i$ with $\{X_1,\ldots,X_f\}$ satisfies (2).

(2)$\Rightarrow$(1): 
For each $i \in \{1,\ldots, f\}$,  any vector in $X_i$ has 
$k$ of entries $\pm 1$, and remaining entries
 are $0$ because $X$ is in $\Omega_{d,k}$.

For each $i \in \{1,\ldots, f\}$, we define the matrix $W_i=[v_1,\ldots, v_d]$ by the vectors $X_i=\{\pm v_1,\ldots, \pm v_d\}$. 
Since $A(X)=\{\pm \sqrt{a}, 0,-k\}$, the entries of $W_i W_j^T$ are $0,\pm \sqrt{a}$ for distinct $i,j$.  
Thus $W_1,\ldots,W_f$ form MQUWM for desired parameters. 
\end{proof}

For MUWM, we obtain the following characterization. 
\begin{proposition}\label{prop:MUWM}
Let $f$, $d$, $k$ be positive integers such that $f\geq2$. 
The existence of the following are equivalent.
\begin{enumerate}
\item a set $\{W_1,\ldots,W_f\}$ of mutually unbiased weighing matrices of wight $k$,  
\item a nonempty subset $X\subset  \sqrt{k}S^{d-1}$ with the property that $A(X)=\{\pm \sqrt{k} ,0,-k\}$ and there exists a cross polytope decomposition $\{X_0,X_1,\ldots,X_f\}$ of $X$.
\end{enumerate}
\end{proposition}
\begin{proof}
(1)$\Rightarrow$(2): 
Let  $X_i=S(W_i)\cup S(-W_i)$ for $i=1,\ldots,f$ and $X_0=S(\sqrt{k}I)\cup S(-\sqrt{k}I)$. 
Then $X=\cup_{i=0}^f X_i$ with $\{X_0,\ldots,X_f\}$ satisfies (2).

(2)$\Rightarrow$(1): 
After the transformation $X_0$ to be $\{\pm \sqrt{k}e_1,\ldots,\pm \sqrt{k}e_d\}$ 
 any vector in $X_i$ has 
$k$ of entries $\pm 1$.
The rest of the argument is same as Proposition~\ref{prop:MQUWM}
\end{proof}

\section{MQUWM for parameters $(d,d,d/2,2d)$}\label{sec:code}
In this section we give an upper bound and a maximal example of MQUWM for parameters $(d,d,d/2,2d)$.

Define $\psi:\mathbb{Z}_2^d\rightarrow \{1,-1\}^d$ as a map such that $\psi((x_i)_{i=1}^d)=((-1)^{x_i})_{i=1}^d$
\begin{theorem}\label{thm:bound1}
Let $W=\{W_1,\ldots, W_f\}$ be a set of mutually quasi-unbiased 
weighing matrices for parameters $(d,d,d/2,2d)$. Then we have 
\[
f \leq 
d.
\]
\end{theorem}
\begin{proof}
Let $\mathcal{C}$ be the set of preimage of $\psi$ for all the elements of $S(W_i)\cup S(-W_i)$ ($1\leq i \leq f$).
Letting $\alpha(z)=2f d(1-\frac{2z}{d})(1-\frac{z}{d})(1-\frac{2z}{d+\sqrt{2d}})(1-\frac{2z}{d-\sqrt{2d}})$ be the annihilator polynomial of $\mathcal{C}$ and $K_k(z)$ the Krawtchouk polynomial of degree $k$, we have the following expansion:
\begin{align*}
\alpha(z)=f(\frac{1}{d}K_0(z)+\frac{1}{d}K_1(z)+\frac{8}{d^2}K_2(z)+\frac{6}{d(d-2)}K_3(z)+\frac{6}{d^2(d-2)}K_4(z)).
\end{align*}
Thus the linear programming bound \cite[Theorem~5.23]{D} shows that $f\leq d$.
\end{proof}

We use the linear codes over $\mathbb{Z}_2$ or $\mathbb{Z}_4$ to obtain MQUWM.
First we use the linear code over $\mathbb{Z}_2$ which contains the first order Reed-Muller code $RM(1,m)$ to obtain MQUWM, see \cite{HV} for the Reed-Muller code. 
\begin{lemma}\label{lem:code}
Let $C$ be a binary linear code of length $d=2^m$ for a positive integer $m$.
Assume that the set of weights of $C$ is $\{0,d/2\pm a,d/2,d\}$ and $C$ has the first order Reed-Muller code $RM(1,m)$ as a subcode.
Let $\{u_1,\ldots,u_f\}$ be a complete set of representative of $C/RM(1,m)$.
Then $\{\psi(u_1+RM(1,m)),\ldots,\psi(u_f+RM(1,m))\}$ provides a cross polytope decomposition of $\psi(C)$ with the inner product set $\{\pm 2a,0,-d\}$.
\end{lemma}
\begin{proof}
Note that, for each codewords $x,y\in \mathbb{Z}_2^d$, the Hamming distance of $x$ and $y$ is $j$ if and only if $\langle \psi(x),\psi(y)\rangle$ is $d-2j$.
By the assumption of weights of $C$, $A(\psi(C))=\{\pm 2a,0,-d\}$. 

Let $D$ be the first order Reed-Muller code $RM(1,m)$. 
Since $\{u_1,\ldots,u_f\}$ is a complete set of representative  of $C/D$, 
$\{\psi(u_i+D)\mid i=1,\ldots,f \}$ provides a partition of $\psi(C)$.
Since  the Reed-Muller code has the weights $\{0,d/2,d\}$, 
each $\frac{1}{\sqrt{d}}\psi(u_i+D)$ for $1\leq i\leq f$ forms a cross polytope and the inner products between vectors in different $\psi(u_i+D)$'s are in $0,\pm 2a$.

Therefore $\{\psi(u_i+D)\mid i=1,\ldots,f \}$ provides a cross polytope partition of $\psi(C)$ with the inner product set $\{\pm 2a,0,-d\}$.
\end{proof}

Next we use the linear code over $\mathbb{Z}_4$ which contains the first order $\mathbb{Z}_4$-Reed-Muller code $ZRM(1,m)$ to obtain MQUWM, see \cite{HV} for the $\mathbb{Z}_4$-Reed-Muller code. 
Define $\phi:\mathbb{Z}^4\rightarrow \mathbb{Z}_2^2$ to be $\phi(0)=(0,0),\phi(1)=(0,1),\phi(2)=(1,1),\phi(3)=(1,0)$.
This map is extended componentwise to a map, also denoted by $\phi$, from $\mathbb{Z}_4^d$ to $\mathbb{Z}_2^{2d}$.
The map  $\phi$ is called the Gray map.
\begin{lemma}\label{lem:z4code}
Let $C$ be a $\mathbb{Z}_4$-linear code of length $d=2^m$ for a positive integer $m\geq 2$.
Assume that the set of $(n_0(x)-n_2(x))+\sqrt{-1}(n_1(x)-n_3(x))$ for $x\in C$ is 
$$\{\pm d, \pm\sqrt{-1}d,0, \pm b,\pm\sqrt{-1}b\}$$ 
for $0<b<d$
and $C$ has the first order $\mathbb{Z}_4$-Reed-Muller code $ZRM(1,m)$ as a subcode.
Let $\{u_1,\ldots,u_f\}$ be a complete set of representative of $C/ZRM(1,m)$.
Then $\{\psi\circ\phi(u_1+ZRM(1,m)),\ldots,\psi\circ\phi(u_f+ZRM(1,m))\}$ provides a cross polytope decomposition of $\psi(C)$.
\end{lemma}
\begin{proof}
Note that, for codewords $x,y\in \mathbb{Z}_4^d$, $x-y$ has $(n_i(x-y))_{i=0}^3$ if and only if $\langle \psi\circ\phi(x),\psi\circ\phi(y)\rangle$ is $2(n_0(x-y)-n_2(x-y))$.
Thus $A(\psi\circ\phi(C))=\{\pm2b,0,-2d\}$ holds.

Let $D$ be the first order $\mathbb{Z}_4$-Reed-Muller code $ZRM(1,m)$.
Since $\{u_1,\ldots,u_f\}$ is a complete set of representative of $C/D$, 
$\{\psi\circ\phi(u_i+D)\mid i=1,\ldots,f \}$ provides a partition of $\psi\circ\phi(C)$.
Since each $\frac{1}{\sqrt{2d}}\psi\circ\phi(u_i+D)$ forms a cross polytope and the inner products between vectors in different $\psi\circ\phi(u_i+D)$'s are in $0,\pm 2b$.

Therefore $\{\psi\circ\phi(u_i+D)\mid i=1,\ldots,f \}$ provides a cross polytope decomposition of $\psi\circ\phi(C)$ with the inner product set $\{\pm2b,0,-2d\}$.
\end{proof}

We construct the maximal example meeting in Theorem~\ref{thm:bound1} for $d=2^{2t+1}$ and $t$ is a positive integer, which is an affirmative answer to Conjecture 23 in \cite{BKR}.
\begin{theorem}
There exists mutually quasi-unbiased weighing matrices for parameters $(d,d,d/2,2d)$ attaining the bound in Theorem~\ref{thm:bound1}, where $d=2^{2t+1}$.
\end{theorem}
\begin{proof}
We provide two constructions of MQUWM.
(1) Apply Lemma~\ref{lem:code} and Proposition~\ref{prop:MQUWM} to the code generated by the extended code of dual of $\mathcal{B}(2,2t+1)$ and the all-ones vector.

(2) Apply Lemma~\ref{lem:z4code} and Proposition~\ref{prop:MQUWM} to the $\mathbb{Z}_4$-Kerdock code of  length $d$.
\end{proof}
\begin{remark}
The Gray image of $\mathbb{Z}_4$-Kerdock code is no longer linear over $\mathbb{Z}_2$, but it still constructs MQUWM.
\end{remark}

% \begin{example}
% Quasi-unbiased weighing matrices given in the subsection provides a cross polytope partition of the spherical code corresponding to the BCH code.
% The BCH code with binary relations obtained from Hamming distance forms an association scheme.  
% Indeed, the number of Hamming distances for any distinct elements in $\mathcal{C}^\perp$ is three and minimum distance of  $\mathcal{C}^\perp$is five,  $\mathcal{C}^\perp$ with binary relations determined from Hamming distances forms a symmetric association scheme \cite[Theorem 5.25]{D}. 
% Thus the example $\hat{\mathcal{C}^\perp}$ with the partition $\{\hat{u_i}+\mathcal{D}^\perp\mid i=1,\ldots,d  \}$ satisfies the assumption of Theorem~\ref{thm:cp}.
% \end{example}

\section{MQUWM for parameters $(d,2,4,1)$} \label{sec:4}
In the present section, we show an upper bound for 
the number of  MQUWM for parameters $(d,2,4,1)$, and give 
examples attaining the bound. 

\begin{theorem}\label{thm:bound2}
Let $W=\{W_1,\ldots, W_f\}$ be a set of  mutually quasi-unbiased 
weighing matrices for parameters $(d,2,4,1)$. Then we have 
\[
f \leq 
d-1.
\]
\end{theorem}
\begin{proof}
Every row vector of $W_i$ is in $(\pm 1, \pm 1, 0,\ldots,0)^P$ whose size $2d(d-1)$. This clearly shows the theorem.  
\end{proof}
There does not exist a weighing matrix of weight $2$ for odd order \cite{BKR13,CRS86}. 
For even order, we have a set of  MQUWM meeting the upper bound in Theorem~\ref{thm:bound2}. 
\begin{theorem} \label{thm:deco}
Suppose $d$ is even. 
Then there exists a set of $d-1$ mutually quasi-unbiased 
weighing matrices for parameters $(d,2,4,1)$. 
\end{theorem}
\begin{proof} 
By Lemma~\ref{lem:maxframes} (2), the roots $(\pm 1, \pm 1, 0, \ldots, 0)^P$ of the $D_d$ lattice  are decomposed into 
$d-1$ disjoint $2$-frames. 
By Proposition~\ref{prop:MQUWM}, we obtain MQUWM for parameters $(d,2,4,1)$.
% For a 2-frame $F=\{\pm v_1, \ldots, 
% \pm v_d\}$ in the $D_d$ lattice, the matrix $M_F=[v_1, \ldots, v_d]$ is a weighing matrix  of
% weight $2$. For two disjoint $2$-frames $F_1$, $F_2$ in the lattice, $M_{F_1} M_{F_2}$ has only three entries $\pm 1,0$, and 
% it satisfies  
% \[
% (M_{F_1}M_{F_2})(M_{F_1}M_{F_2})^T=4I.
% \]
% Therefore we obtain a set of $d-1$ MQUWM for parameters $(d,2,4,1)$ from $d-1$ disjoint $2$-frames in the $D_d$ lattice. 
\end{proof}

\section{MUWM of weight $4$}\label{subsec:frame}
In the present section, we introduce the relationship between MUWM of weight 4 and disjoint $2$-frames in a root lattice, and determine the maximum number of MUWM of weight 4 for any order. 

By Proposition~\ref{prop:MUWM}, 
the existence of MUWM of order $d$ and weight $4$ is equivalent to 
that of a finite subset $X$ in  $2S^{d-1}$ such that $A(X)=\{2,0,-2,-4\}$ and 
$X$ has a cross polytope decomposition. The set $(1/\sqrt{2}) X$ is identified with 
a subset of roots in a root lattice, and the cross polytopes correspond to disjoint $2$-frames.   
If we determine a root lattice of rank $d$ which has the maximum 
number of disjoint $2$-frames for fixed $d$, then we can obtain the maximum size of a set of 
MUWM of order $d$ and weight $4$.  

The following is the main theorem in the present section. The value $m$ means the maximum size of a set of MUWM of order $d$ and weight $4$. 
\begin{theorem}
 The following root lattices have the maximum number $m$ of  disjoint $2$-frames for a fixed
rank $d$. 
\[
\begin{array}{c|cccccc}
d &5 &8  &9&11          &13                      &\text{even $d \geq 4 (d\ne 8)$}\\ \hline  
m &0 &14 &0&2           &4                       &d-2\\ \hline 
\text{{\rm lattice}} 
  &-&E_8&-&D_4\perp E_7&D_6\perp E_7  &D_d(d=16,E_8\perp E_8)
\end{array}
\]
\[
\begin{array}{c|c}
d&\text{odd $d \geq 15 $, $d=7$}\\ \hline
m & 8 \\ \hline
\text{{\rm lattice}}&
(\perp_{\text{$a$ copies}} E_7 )
\perp 
(\perp_{\text{$b$ copies}} E_8 )
\perp 
(\perp_i
(\perp_{\text{$t_i$ copies}}D_{d_i}))\\ 
&\text{even } d_i \geq 10, d=7a+8b+\sum_i t_id_i\quad (a \ne 0)
\end{array}
\] 
For odd $d\geq 17$, we have the lattice $E_7 \perp D_{d-7}$ giving $m=8$ in the above table. 
\end{theorem}
\begin{proof}
For each rank $d$, we consider possible irreducible components 
of a root lattice. By Lemmas~\ref{lem:red} and \ref{lem:maxframes}, we obtain
the maximum number of disjoint $2$-frames.  
\end{proof}

\noindent
\textbf{Acknowledgments.}
We would like to thank the two anonymous refrees for their  valuable comments for the first version of this paper. 
The first author was supported by JSPS KAKENHI Grant Number 25800011.

%%%%%%%%%%%%%%%%%%%%%%%%%%%%%%%%%%%%%%%%%%%%%%%%%%%%%%%%%%%%%%%%%%%%%%%%%%%%%%%%%

\end{document}